\documentclass[a4paper]{article}

\usepackage{style}

\begin{document}

\title{Point compression for the trace zero subgroup\\ over a small degree extension field\thanks{This article appeared in {\it Designs, Codes and Cryptography}, the final publication is available at \url{http://link.springer.com/article/10.1007\%2Fs10623-014-9921-0}.}}

\author{Elisa Gorla}
\affil{\small Institut de math\'ematiques, Universit\'e de Neuch\^atel, Rue Emile-Argand 11, 2000 Neuch\^atel, Switzerland, \small {\tt elisa.gorla@unine.ch}}
\author{Maike Massierer}
\affil{\small Mathematisches Institut, Universit\"at Basel, Rheinsprung 21, 4051 Basel, Switzerland, {\tt maike.massierer@unibas.ch}}

\date{}

\maketitle

\begin{abstract}
\noindent {\bf Abstract}\quad Using Semaev's summation polynomials, we derive a new equation for the $\Fq$-rational points of the trace zero variety of an elliptic curve defined over $\Fq$. Using this equation, we produce an optimal-size representation for such points. Our representation is compatible with scalar multiplication. We give a point compression algorithm to compute the representation and a decompression algorithm to recover the original point (up to some small ambiguity). The algorithms are efficient for trace zero varieties coming from small degree extension fields. We give explicit equations and discuss in detail the practically relevant cases of cubic and quintic field extensions.


\end{abstract}

\section{Introduction} \label{sec:intro}

% One of the most important primitives in public key cryptography is the discrete logarithm problem (DLP). Given a cyclic group $G = \langle g \rangle$ and an element $h \in G$, it requires computing a positive integer $d$ such that $g^{d} = h$. While multiplicative groups of finite fields are widely used in this setting, (hyper)elliptic curves over finite fields are known to be the strongest instance of the DLP in the cryptographic context.

% The trace zero variety captures the hardness of the (hyper)elliptic
% curve discrete logarithm problem over extension fields. It arises
% from the subgroup of points of trace zero of a (hyper)elliptic curve
% by Weil restriction, and it carries the structure of that group.
Given a (hyper)elliptic curve defined over $\Fq$ and a field extension
$\Fq|\Fqn$, consider the $\Fqn$-rational points of trace zero. They
form a subgroup of the group of $\Fqn$-rational points of the curve,
and can be realized as the $\Fq$-rational points of an abelian variety
built by Weil restriction from the original curve, called the trace
zero variety. The trace zero subgroup was first proposed for use in
cryptography by Frey \cite{frey-99}, and further studied by Naumann \cite{naumann}, Weimerskirch \cite{weimerskirch}, Blady \cite{blady}, Lange
\cite{lange-phd,lange-04},
Avanzi--Cesena \cite{avanzi-cesena-07,cesena-10}, and Diem-Scholten \cite{diem-scholten}. Trace zero subgroups are interesting because they allow
efficient arithmetic, due to a speed-up of the standard scalar
multiplication using the Frobenius endomorphism. This is analogous to
the use of endomorphisms to speed up scalar multiplication on Koblitz
curves (see \cite{koblitz-91}) and GLV--GLS curves (see
\cite{glv,gls}), which are the basis for several recent implementation
speed records for elliptic curve arithmetic (see
\cite{longa-sica-12,faz-hernandez-13,bos-costello-hisil-lauter-13}).

The trace zero subgroup is of interest in the context of pairing-based cryptography. 
Rubin and Silverberg have shown in \cite{rubin-silverberg-02,rubin-silverberg-09} 
that the security of pairing-based cryptosystems can be
improved by using abelian varieties of dimension greater than one in
place of elliptic curves. Jacobians of hyperelliptic curves and trace zero
varieties are therefore the canonical examples for such applications.
%More precisely, over a field of characteristic $3$ the known examples
%of groups with highest security parameter (i.e., 7.5) come from trace zero
%subgroups relative to 
%a field extension $\F_{q^5}|\Fq$, while over a binary field trace zero
%subgroups relative to a cubic field extension offer better performance
%and equivalent security to supersingular elliptic curves over cubic
%extension fields (see \cite{cesena-10} for a detailed discussion). 
%In view of these applications, a detailed assessment of the complexity of the discrete logarithm problem (DLP)
%in the trace zero subgroup becomes necessary. In this paper, we briefly discuss the cases $n=3,5$. We come to the conclusion that there are no
%security issues for $n=3$, and that the cover and decomposition attacks on the DLP for $n=5$ can be avoided by imposing some extra
%conditions. Moreover, we explain why the known index calculus attacks on the DLP in the trace zero subgroup do not threaten the security of
%pairing-based cryptosystems involving this group.

Since the trace zero subgroup is contained in the group of
$\Fqn$-rational points of the (Jacobian of the) curve, the DLP in the
trace zero subgroup is at most as hard as the DLP in the curve. It is
easy to show that in fact the DLP's in the two groups have the same
complexity. From a mathematical point of view therefore, trace zero
variety cryptosystems may be regarded as the (hyper)elliptic curve analog of
torus-based cryptosystems such as LUC \cite{luc}, Gong--Harn
\cite{gong-harn-99}, XTR \cite{xtr}, and CEILIDIH \cite{rubin-silverberg-03}. 

The hardness of the discrete logarithm problem in a group is closely connected with the size of the representation of the group elements. Usually, the hardness of the DLP is measured as a function of the group size. However, for practical purposes, the comparison with the size of the representation of group elements is a better indicator, since it quantifies the storage and transmission costs connected with using the corresponding cryptosystem. Therefore, in order to make the comparison between DLP complexity and group size a fair one, we are interested in a compact representation that reflects the size of the group. Such an optimal-size representation consists of $\log_2 N$ bits, where $N$ is the size of the group. See also \cite{gorla-11} for a discussion on the significance of compact representations.

An optimal-size representation for elliptic curves is well-known. In the cryptographic setting, it is standard procedure to represent an elliptic curve point by its $x$-coordinate only, since the $y$-coordinate can easily be recomputed, up to sign, from the curve equation. If desired, the sign can be stored in one extra bit of information. Representing a point via its $x$-coordinate gives an optimal representation for the elements of the group of $\Fqn$-rational points of an elliptic curve: Each of the approximately $q^n$ points can be represented by one element of $\Fqn$, or $n$ elements of $\Fq$ after choosing a basis of the field extension. Notice moreover that storing the sign of the $y$-coordinate is unnecessary, since this representation is compatible with scalar multiplication of points: For any $k\in\Z$, the $x$-coordinates of the points $kP$ and $-kP$ coincide.

The trace zero variety of an elliptic curve with respect to a {\it prime} extension degree $n$ has dimension $n-1$, and we are interested in the $\Fq$-rational points. Hence, an optimal representation should have $\log_2 q^{n-1}$ bits, or consist of $n-1$ elements of $\Fq$. For practical purposes, it is important that the representation can be efficiently computed (``compression'') and that the original point can be easily recovered, possibly up to some small ambiguity, from the representation (``decompression'').
Naumann \cite{naumann}, Rubin--Silverberg \cite{rubin-silverberg-02,rubin-silverberg-04,silverberg-05}, and Lange \cite{lange-04} propose compact representations with compression and decompression algorithms for genus 1 and genus 2 curves, respectively. The work by Eagle--Galbraith--Ong \cite{eagle-galbraith-ong-11} on point compression methods for Koblitz curves is also related. 

In this paper, we concentrate on extension fields of degree $n=3$ or $5$. This is due to the fact that an index calculus attack \cite{gaudry-09} and a cover attack \cite{diem-ghs,diem-06,diem-kochinke} apply to $T_n$, making it vulnerable for large values of $n$. In this work we briefly discuss these attacks and come to the conclusion that there are no security issues for $n=3$. For $n=5$ the cover attacks can be avoided by imposing extra conditions, and the known index calculus attacks do not threaten the security of pairing-based cryptosystems involving trace zero subgroups of supersingular curves.

The main purpose of this paper is introducing a new representation for the points on the trace zero variety of an elliptic curve. The compression and decompression algorithms are more efficient than that of \cite{silverberg-05}, and points are recovered with smaller ambiguity. In addition, our representation is (to the extent of our knowledge) the only one that is compatible with scalar multiplication of points, which is the only operation needed in Diffie--Hellman-based cryptographic protocols.

The paper is structured as follows: In Section \ref{sec:standardrep}, we fix the notation, give the relevant definitions, and briefly recall the standard representation for points on the trace zero variety. We also discuss the simple case of the trace zero variety for a quadratic field extension. Using Semaev's summation polynomials, in Section \ref{sec:eqn1} we derive a single equation whose $\Fq$-solutions describe the $\Fq$-points of the trace zero variety, up to a few well-described exceptions (see Lemma~\ref{lemma:eqn} and Proposition~\ref{prop:eqn}).
In Section \ref{sec:eqn2}, using the equation that we produced in the previous section, we propose a new representation for the points on the trace zero variety. The size of the representation is optimal, and we give efficient compression and decompression algorithms. In Sections \ref{sec:eqnsdegree3} and \ref{sec:eqnsdegree5} we analyze in detail what our method produces for the cases $n=3$ and $5$. We give explicit equations and concrete examples computed with Magma and comment on security issues for these parameters. It is generally agreed that 3 and 5 are the practically relevant extension degrees in the case of elliptic curves (see e.g.\ \cite{lange-04}).

\section{Preliminaries} \label{sec:standardrep}

Let $\Fq$ be a finite field with $q$ elements, 
%of characteristic not equal to $2$ or $3$, 
and let $E$ be an elliptic curve defined over $\Fq$ by an affine Weierstra\ss\ equation.
%defined by the affine equation $$E : y^2 = x^3 + Ax + B,$$ with coefficients $A, B \in \Fq$. 
We consider the group $E(\Fqn)$ of $\Fqn$-rational points of $E$ for field extensions of {\it prime} degree $n$. The group operation is point addition, and the neutral element is the point at infinity, denoted by $\O$. We denote indeterminates by lower case letters and finite field elements by upper case letters.

\begin{definition}
The Frobenius endomorphism on $E$ is defined by $$\varphi : E \rightarrow E,~ (X,Y) \mapsto (X^q, Y^q), ~\O \mapsto \O.$$ One can define a trace map $$\Tr : E(\Fqn) \mapsto E(\Fq), ~P \mapsto P + \varphi(P) + \varphi^2(P) + \ldots + \varphi^{n-1}(P),$$ relative to the field extension $\Fqn|\Fq$. The kernel of the trace map is the {\it trace zero subgroup} of $E(\Fqn)$, which we denote by $T_n$. 

By the process of Weil restriction, the points of $T_n$ can be viewed as the $\Fq$-rational points of an abelian variety $V$ of dimension $n-1$ defined over $\Fq$. $V$ is called the {\it trace zero variety}. 
\end{definition}

In trace zero subgroups, arithmetic can be made more efficient by using the Frobenius endomorphism, following a similar approach to Koblitz curves and GLV--GLS curves. They turn out to be extremely interesting in the context of pairing-based cryptography, where they achieve the largest security parameters in some cases, as discussed in \cite{rubin-silverberg-02,rubin-silverberg-09,avanzi-cesena-07,cesena-10}.

It is easy to show that the DLP in $E(\Fqn)$ is as hard as the DLP in $T_n$. An explanation is given in \cite{granger-vercauteren-05} for the analogous case of algebraic tori: The trace maps a DLP in $E(\Fqn)$ to a DLP in $E(\Fq)$. By solving it in the smaller group, the discrete logarithm is obtained modulo the order of $E(\Fq)$. The remaining modular information required to compute the full discrete logarithm comes from solving a DLP in $T_n$. A formal argument, which applies to any short exact sequence of algebraic groups, is given in~\cite{galbraith-smith-06}.

\begin{proposition}
Consider the exact sequence
$$ 0 \longrightarrow T_n \longrightarrow E(\Fqn) \overset{\Tr}{\longrightarrow} E(\Fq) \longrightarrow 0. $$ 
Then solving a DLP in $E(\Fqn)$ has the same complexity as solving a DLP in $T_n$ and a DLP in $E(\Fq)$.
\end{proposition}

In the conclusions of \cite{avanzi-cesena-07}, large bandwidth is mentioned as the only drawback of using trace zero subgroups in pairing-based cryptography. In this paper we solve this problem by finding an optimal representation for the elements of the trace zero subgroup. 

\begin{definition}
Let $G$ be a finite set. A {\it representation} for the elements of $G$ is a bijection between $G$ and a set of binary strings of fixed length $\ell$. Equivalently, it is an injective map from $G$ to $\F_2^{\ell}$.
A representation is {\em optimal} if $|G|\sim 2^{\ell}$, i.e.\ if we need approximately $\log_2 |G|$ bits to represent an element of $G$.

Abusing terminology, in this paper we call representation a map from $G$ to $\F_2^{\ell}$ with the property that an element of $\F_2^{\ell}$ has at most $d$ inverse images, for some small fixed $d$. In this case, we say that the representation {\it identifies} classes of at most $d$ elements, namely those that have the same representation. Notice that the number of classes is about $|G|/d \sim |G|$, if $d$ is a small constant.
\end{definition}

\begin{remark}
Since the elements of a finite field $\Fq$ can be represented via binary strings of length $\log_2 q$, a representation for $G$ can be given via a bijection between $G$ and a subset of $\F_q^m$, for some $m$ and some prime power $q$. Such a representation is optimal if and only if $|G|\sim q^m$. 
\end{remark}

Representing points of an elliptic curve via their $x$-coordinate is a standard example of optimal representation.

\begin{example}
It is customary to represent a point $(X,Y)\in E(\Fqn)$ via its $x$-coordinate $X\in\Fqn$. The $y$-coordinate can then be recovered, up to sign, from the curve equation. If desired, the sign can be stored in one extra bit of information. Such a representation is optimal, since by Hasse's Theorem 
$|E(\Fqn)|\sim q^n$.
\end{example}

The representation from the previous example identifies pairs of points, since $P$ and $-P$ have the same $x$-coordinate. We often say that the $x$-coordinate is a representation for the {\it equivalence class} consisting of $P$ and $-P$. The representation that we propose in this paper identifies a small number of points as well. Before we discuss our representation, we notice that representing a point $P\in T_n$ via its $x$-coordinate is no longer optimal.

\begin{remark}\label{rmk:nonopt}
Since a point $P=(X,Y)\in T_n$ is an element of $E(\Fqn)$, we can represent $P$ via $X\in\Fqn$. Choosing an $\Fq$-basis of $\Fqn$, we can represent $X\in\Fqn$ as an $n$-tuple $(X_0,\ldots,X_{n-1})\in\F_q^n$. 
Representing $P\in T_n$ as $X\in\Fqn$ or as $(X_0,\ldots,X_{n-1})\in\F_q^n$ however is not optimal, since $|T_n|\sim q^{n-1}$.
\end{remark}

In this paper we find a representation for the elements of $T_n$, via $n-1$ coordinates in $\Fq$. Our representation is not injective, but it identifies a small number of points. Our approach is the following: We start from the representation of $P\in T_n$ as an $n$-tuple $\rho(P)=(X_0,\ldots,X_{n-1})\in\F_q^n$, and write an equation in $\Fq[x_0,\ldots,x_{n-1}]$ which vanishes on $\rho(P)$ for all $P\in T_n$. This allows us to drop one coordinate of $\rho(P)$ and reconstruct it using the equation.
Therefore, we can represent elements of $T_n$ via $n-1$ coordinates in $\Fq$, which is optimal. 

We now fix some notation that we will use when writing explicit equations in Sections~\ref{sec:eqn2}, \ref{sec:eqnsdegree3}, and \ref{sec:eqnsdegree5}.
Let $\Fq$ be the finite field with $q$ elements, $n$ a prime. %Thanks to Proposition~\ref{prop:n=2}, we may assume that $n>2$.
%Let $E$ be a smooth elliptic curve defined over $\Fq$ by an affine Weierstrass equation. We consider the group $E(\Fqn)$ of $\Fqn$-rational
%points of $E$, and its trace zero subgroup $T_n$. $T_n$ is the group of $\Fq$-rational points of an abelian variety $V$, which we call the trace
%zero variety of $E$ with respect to the field extension $\Fq\subset\Fqn$.
For the sake of concreteness, we assume that $n\mid q-1$. Due to its simplicity, we always consider this case when writing explicit equations. All of our arguments however work for any $n$ and $q$, see also Remark~\ref{rmk:anynq}. If $n\mid q-1$, thanks to Kummer theory we can write the extension field as $$\Fqn = \Fq[\zeta]/(\zeta^n - \mu),$$ where $\mu$ is not an $n$-th power in $\Fq$. Where necessary, we take $1, \zeta, \ldots, \zeta^{n-1}$ as a basis of the field extension.

When doing Weil restriction, we associate $n$ new variables $x_0,\ldots,x_{n-1}$ to the variable $x$. They are related via
%As it is customary when representing points of coordinates $(x,y)$ on an elliptic curve, we drop the $y$ coordinate.  
%In Section~\ref{sec:eqn2} we discuss how to write an equation in $\Fq[x_0,\ldots,x_{n-1}]$ whose $\Fq$-solutions correspond via 
\begin{equation}\label{wrcoords}
x =x_0 + x_1 \zeta + \ldots + x_{n-1} \zeta^{n-1}.\end{equation}
%to the $\Fqn$-solutions of an equation in $\Fq[x]$. The solutions of the univariate equation are the 
%$x$-coordinates of the elements of trace zero of $E(\Fqn)$, together with a few well-described exceptions.
%This equation is obtained via Weil restriction from an equation which is computed using the $n$-th Semaev polynomial, and it 
%Both equations can be written down explicitly for fixed $n, q, \mu$. 
We abuse terminology and use the term {\it Weil restriction} not only for the variety, but also for the process of 
writing equations for the Weil restriction. In particular for us, Weil restriction is a procedure that can be applied to a polynomial defined over $\Fqn$ and results in $n$ polynomials with $n$ times as many variables, and coefficients in $\Fq$.

\begin{remark}\label{rmk:anynq}
If $n$ does not divide $q-1$, we choose a normal basis $\{\alpha,\alpha^q,\ldots,\alpha^{q^{n-1}}\}$ of $\Fqn$ over $\Fq$ and Weil restriction coordinates 
%\begin{equation}\label{wrcoords}
$$x = x_0\alpha + x_1 \alpha^q + \ldots + x_{n-1} \alpha^{q^{n-1}}.$$
%\end{equation}
%As before, the procedure that we outline in Section~\ref{sec:eqn2} yields an equation in $\Fq[x_0,\ldots,x_{n-1}]$, whose solutions correspond (up to a few well-described exceptions) via equation (\ref{wrcoords}) to the $x$-coordinates of the points of $T_n$ (which are $\Fqn$-rational points of $E$).
\end{remark}

It is easy to show that the case $n=2$ allows a trivial optimal representation for the elements of $T_n$. Hence in the next sections we concentrate on the more interesting case of {\it odd} primes $n$.

\begin{proposition}\label{prop:n=2}
The trace zero subgroup $T_2$ of $E(\F_{q^2})$ can be described as
$$T_2=\{(X,Y)\in E(\F_{q^2}) \mid X \in \Fq, \;Y\notin\Fq\}\cup E[2](\Fq).$$
In particular, representing a point $(X,Y)\in T_2$ by $X \in \Fq$ yields a representation of optimal size. 
\end{proposition}

\begin{proof}
We first prove that $T_2$ is contained in the union of sets on the right hand side of the equality. 
Let $P\in T_2$, $P\neq\O$, so $P=(X,Y)\in E(\F_{q^2})$. If $P\in E(\Fq)$, then $2P=\O$, hence $P\in E[2](\Fq)$. If $P\notin E(\Fq)$, then $(X,Y)=-(X^q,Y^q)$. In particular $X=X^q$, so $X\in\Fq$, which also implies $Y\notin\Fq$.

To prove the other inclusion, observe that by definition $P\in E[2](\Fq)$ satisfies $2P=\O$, so $P\in T_2$. Let $P=(X,Y)\in E(\F_{q^2})$ with $X\in\Fq$, $Y\notin\Fq$. Since $X\in\Fq$, the points $(X,Y)$ and $\varphi(X,Y)=(X,Y^q)$ are distinct points on $E$ which lie on the same vertical line $x-X=0$. Hence $(X,Y)+\varphi(X,Y)=\O$ and $(X,Y)\in T_2$. 
\end{proof}

The next proposition will be useful when writing equations for the $\Fq$-rational points of the trace zero variety. 
For a multivariate polynomial $h$, we denote by $\deg_{x_i}(h)$ the degree of $h$ in the variable $x_i$.

\begin{proposition}\label{prop:zeros}
 Let $h \in \Fq[x_0,\ldots,x_{n-1}]$ be a polynomial with $h(X_0,\ldots,X_{n-1}) = 0$ for all $(X_0,\ldots,X_{n-1}) \in \F_q^n$, and assume that $\deg_{x_i}(h) < q$ for $i \in \{0,\ldots,n-1\}$. Then $h$ is the zero polynomial.
\end{proposition}

\begin{proof}
 Write 
 $$V(h) = \{(X_0,\ldots,X_{n-1}) \in \Fqbar^{\,n} \mid h(X_0,\ldots,X_{n-1}) = 0\}\subseteq\Fqbar^{\,n}$$
 for the zero locus of $h$ over the algebraic closure of $\Fq$ and 
 $$I(V) = \{f \in \Fq[x_0,\ldots,x_{n-1}] \mid f(X_0,\ldots,X_{n-1}) = 0 \text{ for all } (X_0,\ldots,X_{n-1}) \in V \}$$ 
for the ideal of the polynomials vanishing on some $V \subseteq \Fqbar^{\,n}$.
 
First we show that $I(\F_q^n) = J_n$ where $J_n = (x_0^q-x_0,\ldots,x_{n-1}^q-x_{n-1})$. We proceed by induction on $n$. The claim holds for $n=1$, since the elements of $\Fq$ are exactly those elements of $\Fqbar$ that satisfy the equation $x_0^q-x_0$. Assuming that the statement is true for $n-1$, we have
\begin{eqnarray*}
 I(\F_q^n) & = & \bigcap_{(\alpha_0,\ldots,\alpha_{n-1}) \in \F_q^n} (x_0-\alpha_0,\ldots,x_{n-1}-\alpha_{n-1})\\
 & = & \bigcap_{\alpha_0 \in \Fq} \bigcap_{(\alpha_1,\ldots,\alpha_{n-1})\in\F_q^{n-1}}(x_0-\alpha_0,\ldots,x_{n-1}-\alpha_{n-1})\\
  & = & \bigcap_{\alpha_0 \in \Fq} (x_0-\alpha_0,x_1^q-x_1,\ldots,x_{n-1}^q-x_{n-1})\\
  & = & \left(\prod_{\alpha_0 \in \Fq} (x_0-\alpha_0),x_1^q-x_1,\ldots,x_{n-1}^q-x_{n-1}\right)\\
  & = & J_n.
\end{eqnarray*}
 
Now we show that $h=0$. Since $h$ vanishes on $\F_q^n$, we have $\F_q^n \subseteq V(h) \subseteq \Fqbar^{\,n}$, 
which implies $h\in I(V(h)) \subseteq I(\F_q^n) = J_n$. 
The leading terms of $x_0^q-x_0,\ldots,x_{n-1}^q-x_{n-1}$ with respect to any term order are $x_0^q,\ldots,x_{n-1}^q$, 
in particular they are pairwise coprime. Hence the polynomials $x_0^q-x_0,\ldots,x_{n-1}^q-x_{n-1}$ are a Gr\"obner basis of $J_n$. 
Therefore, $h\in J_n$ implies that $h$ reduces to zero using the generators of $J_n$, i.e.\ if we divide $h$ by $x_i^q-x_i$ 
whenever the leading term of $h$ is divisible by $x_i^q$, we must obtain remainder zero when no more division is possible. 
But since $\deg_{x_i}(h) < q$ for all $i$, $h$ is equal to the remainder of the division of $h$ by $x_0^q-x_0,\ldots,x_{n-1}^q-x_{n-1}$, hence $h = 0$.
\end{proof}

\section{An equation for the trace zero subgroup} \label{sec:eqn1}

In this section we use Semaev's summation polynomials \cite{semaev-04} to write an equation for the set of $\Fq$-rational points of the trace zero variety. The equation involves the $x$-coordinates only and will help us in finding a better representation for the elements of the trace zero subgroup. 

Semaev introduced the summation polynomials in the context of attacking the elliptic curve discrete logarithm problem. They give
polynomial conditions describing when a number of points on an elliptic curve sum to $\O$, involving only the $x$-coordinates of the points.

\begin{definition}\label{def:semaev}
Let $\Fq$ be a finite field of characteristic different from $2$ and $3$ and let $E$ be a smooth elliptic curve defined by the affine equation
$$E : y^2 = x^3 + Ax + B,$$ with coefficients $A, B \in \Fq$.

Define the $m$-th summation polynomial $f_m$ recursively by
$$\begin{array}{rcl}
f_3(z_1,z_2,z_3) & = & (z_1 - z_2)^2z_3^2 - 2((z_1+z_2)(z_1z_2 + A) + 2B)z_3 + (z_1z_2-A)^2 - 4B(z_1+z_2) \\
f_m(z_1,\ldots,z_m) & = & \Res_z(f_{m-k}(z_1,\ldots,z_{m-k-1},z),f_{k+2}(z_{m-k},\ldots,z_m,z)) 
\end{array} $$
for $m \geq 4$ and $m-3 \geq k \geq 1$. 
\end{definition}

We briefly recall the properties of summation polynomials that we will need.

\begin{theorem}[\cite{semaev-04}, Theorem~1]\label{thm:semaev}
For any $m \geq 3$, let $Z_1,\ldots, Z_m$ be elements of the algebraic closure $\Fqbar$ of $\Fq$. Then $f_m(Z_1,\ldots,Z_m) = 0$ if and only if there exist $Y_1,\ldots,Y_m \in \Fqbar$ such that the points $(Z_i,Y_i)$ are on $E$ and $(Z_1,Y_1) + \ldots + (Z_m,Y_m) = \O$ in the group $E(\Fqbar)$. Furthermore, $f_m$ is absolutely irreducible and symmetric of degree $2^{m-2}$ in each variable. The total degree is $(m-1)2^{m-2}$.
\end{theorem}

\begin{remark}
Definition~\ref{def:semaev} is the original definition that Semaev gave in \cite{semaev-04}. Semaev polynomials can be defined and computed also over a finite field of characteristic $2$ or $3$. Although the formulas look different, the properties are analogous to those stated in Theorem~\ref{thm:semaev}. Hence all the results that we prove in this paper hold, with the appropriate adjustments, over a finite field of any charasteristic.
\end{remark}

Since the points in $T_n$ are characterized by the condition that their Frobenius conjugates sum to zero, we can use the Semaev polynomial to give an equation only in $x$. It is clear that $(X,Y) \in T_n$ implies $f_n(X,X^q,\ldots,X^{q^{n-1}}) = 0$. The opposite implication has some obvious exceptions.

\begin{lemma}\label{lemma:eqn}
%For any integer $n\geq 2$ we have
%$$\bigcup_{k=0}^{\lfloor\frac{n}{2}\rfloor-1} E[n-2k](\F_q)\subseteq  \{ (X,Y) \in E(\F_{q^n}) \mid f_n(X,X^q,\ldots,X^{q^{n-1}}) = 0 \} 
%\cup \{\O\}.$$
For any prime $n$, let $T_n$ denote the trace zero subgroup associated with the field extension $\Fqn|\Fq$. We have
$$\bigcup_{k=0}^{\lfloor\frac{n}{2}\rfloor-1} (E[n-2k](\F_q)+E[2]\cap T_n)\subseteq  \{ (X,Y) \in E(\F_{q^n}) \mid f_n(X,X^q,\ldots,X^{q^{n-1}}) = 0 \} 
\cup \{\O\}.$$
\end{lemma}

\begin{proof}
 Let $k \in \{0,\ldots,\lfloor \frac{n}{2} \rfloor \}$, and let $P = Q + R$ with $Q \in E[n-2k](\Fq), R \in E[2] \cap T_n$. Then we have
 \begin{eqnarray*}
  && \underbrace{P + \varphi(P) + \ldots + \varphi^{n-2k-1}(P)}_{n - 2k \text{ summands}} + \underbrace{\varphi^{n-2k}(P) - \varphi^{n-2k+1}(P) + \ldots - \varphi^{n-1}(P)}_{2k \text{ summands with alternating signs}}\\
  & = & \underbrace{Q+\ldots+ Q}_{n-2k \text{ summands}} + \underbrace{Q - Q + \ldots - Q}_{2k \text{ summands with alternating signs}} + R + \varphi(R) + \ldots + \varphi^{n-1}(R)\\
  & = & (n-2k)Q + \Tr(R)\\
  & = & \O,
 \end{eqnarray*}
where for the first equality we use that $Q \in E(\Fq)$ and $R \in E[2]$, and for the third equality we use that $Q \in E[n-2k]$ and $R \in T_n$. 
\end{proof}

%In other words, the $x$-coordinates of all the $\Fq$-rational $(n-2k)$-torsion points are zeros of the 
%$n$-th Semaev polynomial evaluated in the Frobenius conjugates of $x$,  for $2\leq n-2k\leq n$.
Notice that the points of the form $P=Q+R$ with $Q\in E[n-2k](\F_q)$ and $R\in E[2]\cap T_n$ 
are not trace zero points if $Q\neq\O$ and $3\leq n-2k\leq n-2$.
For the interesting cases $n=3$ and $5$ we prove that these are the only exceptions. 

\begin{proposition}\label{prop:eqn}
 Let $T_n$ be the trace zero subgroup associated with the field extension $\Fqn|\Fq$. We have
 $$\begin{array}{lcl}
    T_3 & = & \{ (X,Y) \in E(\F_{q^3}) \mid f_3(X,X^q,X^{q^2}) = 0\} \cup \{\O\}  \\
    T_5 \cup (E[3](\F_q)+E[2]\cap T_5) & = & \{ (X,Y) \in E(\F_{q^5}) \mid f_5(X,X^q,\ldots,X^{q^4}) = 0 \} \cup \{\O\}.
 \end{array}$$
\end{proposition}

\begin{proof}
 Let $P = (X,Y) \in E(\F_{q^3})$ with $f_3(X,X^q,X^{q^2}) = 0$. Then by the properties of the Semaev polynomial, there exist $Y_0,Y_1,Y_2 \in \Fqbar$ such that $(X,Y_0) + (X^q,Y_1) + (X^{q^2},Y_2) = \O$. Obviously we have $Y_i = \pm Y^{q^i}, i = 0,1,2$, so $P \pm \varphi(P) \pm \varphi^2(P) = \O$. We have to show that all signs are ``+''. Suppose $P - \varphi(P) + \varphi^2(P) = \O$. By applying $\varphi$, we get $\varphi(P) - \varphi^2(P) + P = \O$. Adding these two equations gives $2P = \O$, implying that $P = -P$, hence $P+\varphi(P)+\varphi^2(P)=\O$. In particular, $P \in T_3$. The rest follows by symmetry.
 
Now let $P = (X,Y) \in E(\F_{q^5})$ with $f_5(X,X^q,\ldots,X^{q^4}) = 0$. Then as before, $P \pm \varphi(P) \pm \varphi^2(P) \pm \varphi^3(P) \pm \varphi^4(P) = \O$. If all signs are ``+'', then $P \in T_5$. We treat all other cases below.
 
 [one minus] Assume $P + \varphi(P) + \varphi^2(P) + \varphi^3(P) -
 \varphi^4(P) = \O$. Applying $\varphi$ to the equation and adding the
 two equations, we get $2\varphi(P) + 2\varphi^2(P) + 2\varphi^3(P) =
 \O$, and by substituting into twice the first equation, $2P =
 \varphi^4(2P)$. Hence $2P \in E(\F_{q^4}) \cap E(\F_{q^5}) =
 E(\F_q)$, so $2P\in E[3](\Fq)$. Now $P = Q + R\in E[6]$ is the sum of $Q\in
 E[3]$ and $R\in E[2]$. We have $Q = -2Q = -2P \in E[3](\Fq)$. From
 the original equation $P + \varphi(P) + \varphi^2(P) + \varphi^3(P) -
 \varphi^4(P) = \O$, we get an analogous equation in $R$, which
 together with $R \in E[2]$ gives $R \in T_5$. 
 
 [two minuses in a row] Assume $P + \varphi(P) + \varphi^2(P) -
 \varphi^3(P) - \varphi^4(P) = \O$. Applying $\varphi^2$ and adding,
 we get $2\varphi^2(P) = \O$, hence $P = -P$ and therefore $P \in
 T_5$. 
 
 [two minuses not in a row] Finally, assume $P + \varphi(P) -
 \varphi^2(P) + \varphi^3(P) - \varphi^4(P) = \O$. Applying $\varphi$
 and adding, we get $2 \varphi(P) = \O$, hence $P = -P$ and therefore
 $P \in T_5$. 
 
 The other cases follow by symmetry, thus proving the claim. 
\end{proof}

\begin{remark}\label{rmk:solns}
In the sequel, we use $f_n$ as an equation for $T_n$. In practice however, for any root $X\in\Fqn$ of $f_n(x,x^q,\ldots,x^{q^{n-1}})$ we need to be able to decide efficiently whether $(X,Y)\in T_n$.

For $n=3$ we only need to check that $Y\in\F_{q^3}$. This guarantees that $(X,Y)\in T_3$, by Proposition~\ref{prop:eqn}.

For $n=5$, by Proposition~\ref{prop:eqn} we have to exclude from the solutions of $f_5=0$ the points $(X,Y)\in E$ such that $Y\notin\F_{q^5}$ and the points of the form $Q+R$ where $\O\neq Q\in E[3](\Fq)$ and $R\in E[2]\cap T_5$. Let $\mathcal{L}$ be the set of the $x$-coordinates of the elements $Q+R\in E[3](\Fq)+E[2]\cap T_5$ with $Q\neq\O$. Then $\mathcal{L}$ has cardinality at most $16$. A root $X\in\F_{q^5}$ of $f_5(x,x^q,\ldots,x^{q^4})$ corresponds to a point $(X,Y)\in T_5$ if and only if  $X\notin\mathcal{L}$ and $Y\in\F_{q^5}$. 
\end{remark}

The $x$-coordinates of the points of $T_n$ correspond to zeros of the Weil restriction of the polynomial 
$f_n(x,\ldots,x^{q^{n-1}})$. Since $E$ is defined over $\Fq$, then $f_n(x,\ldots,x^{q^{n-1}})\in\Fq[x]$. 
Therefore, for any $\alpha\in\Fqn$ we have
$$f_n(\alpha,\ldots,\alpha^{q^{n-1}})^q=f_n(\alpha^q,\ldots,\alpha^{q^{n-1}},\alpha)
=f_n(\alpha,\ldots,\alpha^{q^{n-1}}),$$ 
where the second equality follows from the symmetry of the Semaev polynomial. It follows that 
\begin{equation}\label{fqrat}
f_n(\alpha,\ldots,\alpha^{q^{n-1}})\in\Fq~~~\mbox{for all}~\alpha\in\Fqn.
\end{equation}
We use the relations (\ref{wrcoords}) to write equations for the Weil restriction.
Notice that since we are only interested in the $\Fq$-rational points of the Weil restriction, 
we may reduce the equations that we obtain modulo $x_i^q-x_i$ for $i=0,\ldots,n-1$. 
Hence we obtain equations in $x_0,\ldots,x_{n-1}$ of degree less than $q$ in each indeterminate.
Now (\ref{fqrat}) together with Proposition \ref{prop:zeros} implies that the last $n-1$ 
equations are identically zero. 
Therefore, although Weil restriction could produce up to $n$ equations,
by reducing modulo the equations $x_i^q-x_i$ we obtain only one equation at the end.
%This is because a variety of dimension $n-1$ in an $n$-dimensional affine space is defined by 
%exactly one equation.
%, as mentioned before, the trace of an $\Fqn$-rational point is an $\Fq$-rational point. 
%Hence Weil restriction must produce exactly one equation in the new coordinates. 
We denote this new equation by
$$\tilde{f}_n(x_0,\ldots,x_{n-1}) = 0.$$
We stress that its $\Fq$-solutions correspond to the elements of $T_n$, 
together with some extra points described in Lemma~\ref{lemma:eqn} and Proposition~\ref{prop:eqn}. 
In Remark~\ref{rmk:solns} we discussed how to distinguish the extra solutions.
Since we reduce the Weil restriction of $f_n(x,x^q,\ldots,x^{q^{n-1}})$ modulo $x_i^q-x_i$, 
the $q$th powers disappear, and we are left with an equation $\tilde{f}_n$ of the same degree 
as the original Semaev polynomial $f_n$.

%Compared to the straightforward system of equations describing the trace zero variety (as explained in Section \ref{sec:standardrep}, these are $n+1$ equations in the $2n$ indeterminates $x_i$ and $y_i$), we are able to eliminate all $y$-coordinates and $n$ of the equations, leaving us with only one equation in $n$ indeterminates. We pay for this with a slightly higher degree.
Concerning the representation, we now have an equation that is compatible with dropping the $y$-coordinate. It is a natural idea to drop one $X_i$ in order to obtain a compact representation, mimicking the approach of~\cite{naumann,lange-04,silverberg-05}. The decompression algorithm could then use $\tilde{f}_n$ to recompute the missing coordinate. However, since $\tilde{f}_n$ has relatively large degree, this would identify more points than desired. Moreover, the computation of the Weil restriction of the Semaev polynomials requires a large amount of memory. It is already very demanding for $n=5$. We present a modified approach to the problem in the next section.

\section{An optimal representation} \label{sec:eqn2}

As the Semaev polynomials are symmetric in nature, they can be written in terms of the symmetric functions. We write
\begin{equation} \label{eqn:symmetrize}
 f_n(z_1,\ldots,z_n) = g_n(e_1(z_1,\ldots,z_n),\ldots,e_n(z_1,\ldots,z_n)), 
\end{equation}
where $e_i$ are the elementary symmetric polynomials
$$e_i(z_1,\ldots,z_n) = \sum_{1 \leq j_1 < \ldots < j_i \leq n} z_{j_1} \cdot \ldots \cdot z_{j_i},$$
and call $g_n$ the ``symmetrized'' $n$-th Semaev polynomial. The advantage over the original Semaev polynomial is that $g_n$ has lower degree (e.g. 2 instead of 4 for $n=3$, and 8 instead of 32 for $n=5$) and fewer $\Fqn$-solutions, as it respects the inherent symmetry of the sum (i.e.\ where $f_n$ has as solutions all permutations of possible $x$-coordinates, $g_n$ has only one solution, the symmetric functions of these coordinates). See \cite{joux-vitse-12} for how to efficiently compute the symmetrized Semaev polynomials.
In this sense,
\begin{eqnarray} \label{eqn2}
 g_n(s_1,\ldots,s_n) = 0
\end{eqnarray}
also describes the points of $T_n$ via the relations
$$s_i = e_i(x,x^q,\ldots,x^{q^{n-1}}), ~i = 1,\ldots,n.$$
Notice that for $X \in \Fqn$, we have $e_i(X,X^q,\ldots,X^{q^{n-1}}) \in \Fq$. Summarizing, $g_n$ is a polynomial with $\Fq$-coefficients by equation (\ref{eqn:symmetrize}), as well as the polynomials $\tilde{e_i}$ that we obtain by Weil restriction from the symmetric functions in the $q$-powers of $x$:\begin{eqnarray} \label{relres}
 s_i = \tilde{e}_i(x_0,\ldots,x_{n-1}),~ i = 1,\ldots,n.
\end{eqnarray}
Furthermore, we get exactly one new relation per equation (reducing modulo $x_i^q-x_i$ and applying Proposition \ref{prop:zeros}, as before). Hence we
have a total of $n$ equations in the Weil restriction coordinates
describing the symmetric functions. The $q$th powers in the exponents
disappear thanks to the reduction, and each $\tilde{e}_i$ is
homogeneous of degree $i$. A combination of the equations (\ref{eqn2})
and (\ref{relres}) enables us to give a compact representation of the affine points of $T_n=V(\Fq)$. It can be computed with the {\it compression} algorithm, the full point can be recovered (up to some small ambiguity) with the {\it decompression} algorithm.

\bigskip
\noindent
{\bf Compression.}

\smallskip \noindent
{\sc Input:} $P = (X_0,\ldots,X_{n-1},Y_0,\ldots,Y_{n-1}) \in V(\Fq)$

\smallskip \noindent
Compute the symmetric functions of the Frobenius conjugates of $X$:
$$S_i = \tilde{e}_i(X_0,\ldots,X_{n-1}),~ i = 1,\ldots,n-1$$

\smallskip \noindent
{\sc Output:} $(S_1,\ldots,S_{n-1}) \in \F_q^{n-1}$

%\vspace{1cm}
\bigskip
\noindent
{\bf Decompression.}

\smallskip \noindent
{\sc Input:} $(S_1,\ldots,S_{n-1}) \in \F_q^{n-1}$

\smallskip \noindent
Solve $g_n(S_1,\ldots,S_{n-1},t) = 0$ for $t$.\\
For each solution $\tau$, find a solution (if it exists) of the system
\begin{equation} \label{decsys}
\begin{array}{rcl}
S_1 & = & \tilde{e}_1(x_0,\ldots,x_{n-1}) \\
& \vdots & \\
S_{n-1} & = & \tilde{e}_{n-1}(x_0,\ldots,x_{n-1}) \\
\tau & = & \tilde{e}_n(x_0,\ldots,x_{n-1}).
\end{array}
\end{equation}
For the found solution $(X_0^{(j)},\ldots,X_{n-1}^{(j)})$, recompute one of the $y$-coordinates $Y^{(j)}$ belonging to $X^{(j)} = X_0^{(j)} + \ldots + X_{n-1}^{(j)} \zeta^{n-1}$ using the curve equation.\\
If $(X^{(j)},Y^{(j)})\in T_n$, then add $\pm P = (X^{(j)},\pm Y^{(j)})$ and all their Frobenius conjugates to the set of output points.

\smallskip \noindent
{\sc Output:} All points of $T_n=V(\Fq)$ that have $(S_1,\ldots,S_{n-1})$ as compact representation

\begin{remark}
Because of Lemma~\ref{lemma:eqn}, in the last step of the decompression algorithm, for each root $X^{(j)}$ of the polynomial $f_n$ one needs to check that the point $(X^{(j)},Y^{(j)})\in T_n$. This step can in practice be eliminated for $n=3,5$, as discussed in Remark~\ref{rmk:solns}.
\end{remark}

% \bigskip

%For a Zariski-closed (i.e.\ very
For a small set of points, equation (\ref{eqn2}) vanishes when evaluated in the given $S_1,\ldots,S_{n-1}$. For such points $P$, any $t\in\Fq$ solves the equation $g_n(S_1,\ldots,S_{n-1},t)=0$, making the computational effort for decompressing $\compr(P)$ very large. Therefore, our decompression algorithm is not practical for such points. However, for almost all points $P \in V(\Fq)$ the polynomial $g_n(S_1,\ldots,S_{n-1},t)$ has only a small number of roots in $t$ (upper bounded by the degree of $g_n$ in the variable $t$). For our analysis, we assume that we are in the latter case. Since the points of $V(\Fq)$ are described by $g_n(\tilde{e}_1(x_0,\ldots,x_{n-1}),\ldots,\tilde{e}_n(x_0,\ldots,x_{n-1}))$, we have $P \in \decompr(\compr(P))$. The relevant question is how many more points the output may contain. 

First of all, by compressing a point, we lose the ability to distinguish between Frobenius conjugates of points, since for each solution of system (\ref{decsys}), all Frobenius conjugates are also solutions. This can be compared to the fact that when using the ``standard'' compression, we lose the ability to distinguish between points and their negatives. If desired, a few extra bits can be used to remember that information. Alternatively, we can think of working in $T_n$ modulo an equivalence relation that identifies the Frobenius conjugates of each point and its negative. This reduces the size of the group $T_n$ by a factor $2n$, which is a small price to pay considering the amount of memory saved by applying the compression, especially since $n$ is small in practice. Notice also that it is enough to compute one solution of system (\ref{decsys}), since the set of all solutions consists precisely of the Frobenius conjugates of one point. This is because any polynomial in $n$ variables which is left invariant by any permutation of the variables can be written uniquely as a polynomial in the elementary symmetric functions $e_1,\ldots,e_n$.

Now, how many different equivalence classes of points can be output by the decompression algorithm depends only on the degree of $g_n$ in the last indeterminate. For $n=3$ the degree is one and decompression therefore outputs only a single class. As $n$ grows, the degree of the Semaev polynomial also grows, thus producing more ambiguity in the recovery process. This also reflects the growth in the number of extra points which satisfy the equation coming from the Semaev polynomial, as seen in Lemma~\ref{lemma:eqn}.

Notice moreover that there may be solutions $\tau$ of $g_n(S_1,\ldots,S_{n-1},t)=0$ for which system (\ref{decsys}) has no solutions, and that not all the solutions of system (\ref{decsys}) produce an equivalence class of points on the trace zero variety. E.g., if $X\in\Fqn$ satisfies $f_n(X,X^q,\ldots,X^{q^{n-1}})=0$, the corresponding point $P=(X,Y)\in E$ may have $Y\in\F_{q^{2n}}\setminus\Fqn$. In this case $P\notin T_n$.

Since our algorithms are most useful for $n=3$ and $5$, an asymptotic complexity analysis for general $n$ does not make much sense. In fact, it is easy to count the number of additions, multiplications, and squarings in $\Fq$ needed to compute the representation just from looking at the formulas for $s_1,\ldots,s_{n-1}$. We do this for the cases $n=3$ and 5 in Sections \ref{sec:eqnsdegree3} and \ref{sec:eqnsdegree5}, respectively. There, we also discuss the efficiency of our decompression algorithm and how it compares to the approaches of \cite{naumann,silverberg-05}.

\begin{remark}
 In order to compute with points of $T_n$, we suggest to decompress a
 point, perform the operation in $E(\Fqn)$, and compress again the
 result. Since compression and decompression is very efficient, this
 adds only little overhead. In an environment with little storage
 and/or bandwidth capacity, the memory savings of compressed points
 may well be worth this small trade-off with the efficiency of the
 arithmetic. Also notice that scalar multiplication of trace zero
 points in $E(\Fqn)$ is more efficient than scalar multiplication of
 arbitrary points of $E(\Fqn)$, due to a speed-up using the Frobenius
 endomorphism, as pointed out by Frey \cite{frey-99} and studied in
 detail by Lange \cite{lange-phd,lange-04} and subsequently by Avanzi
 and Cesena~\cite{avanzi-cesena-07}.
 
 Our recommendation corresponds to usual implementation practice in the setting of point compression: Even when a method to compute with compressed points is available, it is usually preferable to perform decompression, compute with the point in its original representation, and compress the result. For example, Galbraith-Lin show in \cite{galbraith-lin-09} that although it is possible to compute pairings using the $x$-coordinates of the input points only, it is more efficient in most cases (namely, whenever the embedding degree is greater than 2) to recompute the $y$-coordinates of the input points and perform the pairing computation on the full input points. As a second example, let us consider the following two methods for scalar multiplication by $k$ of an elliptic curve point $P = (X,Y)$ when only $X$ is given:
 \begin{enumerate}
  \item Use the Montgomery ladder, which computes the $x$-coordinate of $kP$ from $X$ only.
  \item Find $Y$ by computing a square root, apply a fast scalar multiplication algorithm to (X,Y), and return only the $x$-coordinate of the result.
 \end{enumerate}
All recent speed records for scalar multiplication on elliptic curves
have been set using algorithms that need the full point $P$, in other
words with the second approach, see e.g.\ \cite{ed25519,longa-sica-12,lambda,faz-hernandez-13}. Timings typically ignore
the additional cost for point decompression, but there is strong
evidence that on a large class of elliptic curves the second approach is
faster. This is the basis for our suggestion to follow the second approach when working with compressed points of $T_n$.
\end{remark}

\section{Explicit equations for extension degree 3} \label{sec:eqnsdegree3}

We give explicit equations for $n=3$, where we write $\F_{q^3} = \Fq[\zeta]/(\zeta^3-\mu)$ and use $1,\zeta,\zeta^2$ as a basis for $\F_{q^3}|\Fq$. For completeness, we start with the standard equations for the trace zero variety (see \cite{frey-99}), although we do not make further use of them in our approach. They describe an open affine part of the trace zero variety (i.e.\ they hold when $x_1, x_2 \neq 0$):
\begin{equation} \label{sys:frey}
\begin{array}{rcl}
  y_0^2 + 2 \mu y_1 y_2 & = & x_0^3 + \mu x_1^3 + \mu^2 x_2^3 + 6 \mu x_0 x_1 x_2 + A x_0 + B\\
  2 y_0 y_1 + \mu y_2^2 & = & 3 x_0^2 x_1 + 3 \mu x_0 x_2^2 + 3 \mu x_1^2 x_2 + A x_1\\
  2 y_0 y_2 + y_1^2 & = & 3 x_0^2 x_2 + 3 x_0 x_1^2 + 3 \mu x_1 x_2^2 + A x_2\\
  x_1 y_2 & = & x_2 y_1.
\end{array}
\end{equation}
The equation that we found in Section~\ref{sec:eqn1} only involves the $x$-coordinate and is
\begin{eqnarray*}
 f_3(x,x^q,x^{q^2}) & = & x^{2q^2+2q} - 2x^{2q^2+q+1} + x^{2q^2+q} - 2x^{q^2+2q+1} -2x^{q^2+q+2}-2Ax^{q^2+q} \\
                                  & &  - 2Ax^{q^2+1} - 4Bx^{q^2} + x^{2q+2} - 2Ax^{q+1} - 4Bx^q- 4Bx + A^2.
\end{eqnarray*}
For Weil restriction, we write $x = x_0 + x_1 \zeta + x_2 \zeta^2$ and get
$$ \begin{array}{rcl}
 x & = & x_0 + x_1 \zeta + x_2 \zeta^2\\
 x^q & = & x_0 + \mu^b x_1 \zeta + \mu^{2b} x_2 \zeta^2 \\
 x^{q^2} & = & x_0 + \mu^{2b} x_1 \zeta + \mu^b x_2 \zeta^2,
\end{array} $$
where $b = \frac{q-1}{3}$. The second and third equalities follow from observing that we can substitute $x_i$ for $x_i^q$ when looking for $\Fq$-solutions. This gives
\begin{equation} \label{eqn:naumann}
\begin{array}{rcl}
  \tilde{f}_3(x_0,x_1,x_2) & = & -3x_0^4 - 12 \mu^2 x_0 x_2^3 - 12 \mu x_0 x_1^3 + 18 \mu x_0^2 x_1 x_2\\
  & &  + 9 \mu^2 x_1^2 x_2^2 - 6Ax_0^2 + 6A \mu x_1 x_2 - 12Bx_0 + A^2.
\end{array}
\end{equation}
The symmetrized third Semaev polynomial is
\begin{equation} \label{eqn3}
 g_3(s_1,s_2,s_3) = s_2^2 - 4s_1s_3 - 4Bs_1 - 2As_2 + A^2 
\end{equation}
and describes the trace zero subgroup via
\begin{equation} \label{sys3}
 \begin{array}{rclcl}
 s_1 & = & x + x^q + x^{q^2} & = & 3x_0 \\
 s_2 & = & x^{1+q} + x^{1+q^2} + x^{q+q^2} & = & 3x_0^2 - 3\mu x_1x_2\\
 s_3 & = & x^{1+q+q^2} & = & x_0^3 - 3 \mu x_0 x_1 x_2 + \mu x_1^3 + \mu^2 x_2^3. 
\end{array} 
\end{equation}
So for compression of a point $(x_0,x_1,x_2,y_0,y_1,y_2)$, we use the coordinates
$$ (s_1,s_2) = (3x_0,3x_0^2 - 3\mu x_1x_2), $$
and for decompression, we have to solve $g_3(s_1,s_2,s_3) = 0$ for
$s_3$, where $g_3$ is given by equation (\ref{eqn3}). Since the
equation is linear in $s_3$, the missing coordinate can be recovered
uniquely, except when $s_1 = 0$. This is the case only for a small set of points. 
Notice moreover that the points $(0,s_2,s_3)$
with $s_2^2-2As_2+A^2=0$ satisfy equation (\ref{eqn3}) for every
$s_3$. %Therefore, they all correspond to points in $T_3$.
%If desired, the coordinates $(s_2,s_3)$ can be used as a
%representation for such points. In any case, 
The only ambiguity in decompression comes from solving system (\ref{sys3}), which yields the Frobenius conjugates $x,x^q,x^{q^2}$ of the original $x$. So for $n = 3$ this gives an optimal representation in our sense.

The following representation is equivalent to the above, but easier to compute. Set
\begin{equation} \label{sys4}
 t_1  =  x_0,~
 t_2  =  x_1x_2,~
 t_3  =  x_1^3 + \mu x_2^3,
\end{equation}
and take $(t_1, t_2)$ as a representation. The relation between the two sets of coordinates is
$$
 s_1  =  3t_1,~
 s_2  =  3t_1^2 - 3 \mu t_2,~
 s_3  =  t_1^3 - 3\mu t_1t_2 + \mu t_3.
$$
In this case, we recover $t_3$ from the equation
$$ -3t_1^4 + 18 \mu t_1^2 t_2 + 9 \mu^2 t_2^2 - 12 \mu t_1t_3 - 12Bt_1 -6At_1^2 + 6A\mu t_2 + A^2. $$
The equation is linear in $t_3$, thus making point recovery unique whenever $t_1 \neq 0$, but the total degree is higher. Compared to the representation $(s_1,s_2)$, fewer operations are needed for compression and for computing the solutions of the system during decompression. Thus, compression and decompression for this variant of the representation are more efficient. We give timings for 10, 20, 40, 60, and 79 bit fields in Table \ref{timings3}, where we see that compression is about a factor 3 to 4 faster and decompression is slightly faster for the second method. Notice that decompression timings are for recomputing the $x$-coordinate only.

All computations were done with Magma version 2.19.3 \cite{magma}, running on one core of an Intel Xeon Processor X7550 (2.00 GHz) on a Fujitsu Primergy RX900S1. Our Magma programs are straight forward implementations of the methods presented here and are only meant as an indication. No particular effort has been put into optimizing them. 

\begin{table}
\caption{Average time in milliseconds for compression/decompression of one point when $n=3$}
\label{timings3}
\begin{tabular}{l|lllll}
\hline\noalign{\smallskip}
$q$ & $2^{10}-3$ & $2^{20}-3$ & $2^{40}-87$ & $2^{60}-93$ & $2^{79}-67$  \\
\noalign{\smallskip}\hline\noalign{\smallskip}
Compression $s_i$ & 0.007 & 0.014 & 0.028 & 0.039 &  0.064\\
Compression $t_i$ & 0.002 & 0.007 & 0.008 & 0.010 & 0.015 \\
Decompression $s_i$ & 0.124 & 0.159 & 0.731 & 0.987 & 1.586\\
Decompression $t_i$ & 0.090 & 0.132 & 0.610 & 0.956 & 1.545\\
\noalign{\smallskip}\hline
\end{tabular}
\end{table}

We give a concrete example, before concluding the section with a more detailed analysis of the efficiency of our algorithms.
\begin{example} \label{ex3}
Let $E$ be the curve $y^2 = x^3 + x + 368$ over $\Fq$, where $q = 2^{79}-67$ is a $79$-bit prime, and $\mu = 3$. The trace zero subgroup of $E(\F_{q^3})$ has prime order of 158 bits. We choose a random point (to save some space, we write only $x$-coordinates)
$${\small 
\begin{array}{c}
P = 260970034280824124824722 + 431820813779055023676698 \zeta + 496444425404915392572065 \zeta^2   \in T_3
\end{array} } $$
and compute
$${\small 
\begin{array}{l} 
 \compr(P)  =  (178447193035157787121145, 159414355696879147312583) \smallskip \\ 
 \decompr(178447193035157787121145, 159414355696879147312583) = \\ 
        ~~~\{ 260970034280824124824722 + 431820813779055023676698 \zeta + 496444425404915392572065 \zeta^2  , \\
        ~~~~   260970034280824124824722 + 318397306102476549147695 \zeta + 124410673032925784958936 \zeta^2 , \\
        ~~~~   260970034280824124824722 + 458707699733097601881649 \zeta + 88070721176787997175041 \zeta^2   \}
\end{array}}$$
where the results of decompression are exactly the Frobenius conjugates of $P$. In our Magma implementation, we solve system (\ref{sys3}) over $\Fq$ similarly to how one would do it by hand, as described below.
%. We first compute $x_0$. Then we distinguish the cases $x_1 = 0$ and $x_1 \neq 0$. In the first case, we check that the values for $x_0$ and $x_1$ do not produce a contradiction in the second equation, and if not, we plug them into the third equation and solve for $x_2$. In the second case, we solve the second equation for $x_2$ (dividing by $x_1$), plug into the third equation, solve for $x_1$, and finally compute the value of $x_2$. In all cases, we get exactly the Frobenius conjugates of the $x$-coordinate of $P$. 
Note that the solutions could also be found by computing the roots of the polynomial $x^3 - s_1x^2 + s_2x - s_3$ over $\F_{q^3}$, but since the system is so simple for $n=3$, solving the system directly is faster in all instances.

When using the second variant of the representation, we compute
$$ {%\small 
\begin{array}{c}
 (t_1,t_2) = (260970034280824124824722, 492721032528256431308437)
\end{array}}$$
and naturally get the same result for decompression by solving system (\ref{sys4}) in a similar way.
\end{example}

\medskip \noindent
{\bf Operation count for representation in the $s_i$.} 
Where possible, we count squarings (S), multiplications (M), and divisions (D) in $\Fq$. We do not count multiplication by constants, since they can often be chosen small (see \cite{lange-04}), and multiplication can then be performed by repeated addition. Compressing a point clearly takes 1S+1M. Decompression requires the following steps.
\begin{itemize}
 \item Evaluating $g_3(s_1,s_2,s_3)$ in the first two indeterminates and solving for the third indeterminate means computing $s_3 = \frac{1}{4s_1}(s_2(s_2-2A) - 4Bs_1 + A^2)$, which takes 1M+1D.
 \item Given $s_1,s_2,s_3,$ we need to solve system (\ref{sys3}) for $x$, or for $x_0,x_1,x_2$. The most obvious way would be to compute the roots of the univariate polynomial $x^3 - s_1x^2 + s_2x - s_3$ over $\F_{q^3}$. Finding all roots of a degree $d$ polynomial over $\Fqn$ takes $O(n^{\log_2 3} d^{\log_2 3} \log d \log(dq^n))$ operations in $\Fq$ using Karatsuba's algorithm for polynomial multiplication (see \cite{gerhard-gathen}). In our case, the degree and $n$ are constants, and hence factoring this polynomial takes $O(\log q)$ operations in $\Fq$. However, since the system is so simple, in practice it is better to solve directly for $x_0,x_1,x_2$ over $\Fq$. We know that the system has exactly three solutions (except in very few cases, where it has a unique solution in $\Fq$, i.e. $x_1 = x_2 = 0$). We get $x_0$ from $s_1$ for free. Assuming that $x_1 \ne 0$ (the special case when $x_0=0$ is easier than this general case), we can solve the second equation for $x_2$, plug this into the third equation, and multiply by the common denominator $27 \mu^3 x_1^3$. In this way, we obtain the equation
\begin{equation*}
  27 \mu^4 x_1^6 + 27 \mu^3( x_0 (s_2 - 2 x_0^2) -  s_3) x_1^3 + \mu^2 (3x_0^2-s_2)^3,
\end{equation*}
which must be solved for $x_1$. The coefficient of $x_1^6$ is a constant, the coefficient of $x_1^3$ can be computed with 1S+1M, and the constant term can then be computed with 1S+1M. Now we can solve for $x_1^3$ with the quadratic formula, which takes 1S and a square root in $\Fq$ for the first value, which will have either no or three distinct cube roots. In case it has none, we compute the second value for $x_1^3$, using only an extra addition, and the three distinct cube roots of this number. This gives a total of 3 values for $x_1$. Finally, we can compute $x_2 = \frac{3 x_0^2 - s_2}{3 \mu x_1}$, which takes 1D for the first, and a multiplication by the inverse of a cube root of unity for the other two values. Altogether, solving system (\ref{sys3}) takes a total of at most 3S+2M+1D, 1 square root, and 2 cube roots in $\Fq$.
\item Finally, for each of the at most $3$ values for $x$, we recompute a corresponding $y$-coordinate from the curve equation and check that it belongs to $\F_{q^3}$. Since these are standard procedures for elliptic curves, we do not count operations for these tasks.
\end{itemize}
Therefore, the decompression algorithm takes at most 3S+3M+2D, one square root, and two cube roots in $\Fq$. The cost of computing the roots depends on the specific choice of the field and on the implementation, but it clearly dominates this computation.

\bigskip \noindent
{\bf Operation count for representation in the $t_i$.} 
In this case, compression takes only 1M. For decompression, we proceed as follows.
\begin{itemize}
 \item Given $t_1$ and $t_2$, we recover $t_3$ from the equation $t_3 = \frac{1}{12 \mu t_1}(-3t_1^4 + (18 \mu t_1^2 + 9 \mu^2 t_2 + 6A \mu)t_2 - 12Bt_1 - 6At_1^2 + A^2)$. This takes 2S+1M+1D.
 \item To solve system (\ref{sys4}), again assuming $x_1 \ne 0$, we have to find the roots of the equation
 $$ x_1^6 - t_3x_1^3 + \mu t_2^3 .$$
 The coefficients of this equation can be computed with a total of 1S+1M. We proceed as above to compute 3 values for $x_1$ using 1S, 1 square root, and 2 cube roots. Finally, we compute $x_2 = \frac{t_2}{x_1}$ using 1D. Thus, solving the system takes a total of at most 2S+1M+1D, 1 square root, and 2 cube roots.
\end{itemize}
In total, decompression takes at most 4S+2M+2D, 1 square root, and 2 cube roots. The cost of this computation is comparable to the decompression using $s_i$. This corresponds to our experimental results with Magma (see Table \ref{timings3}).

\bigskip \noindent
{\bf Comparison with Silverberg's method.}
The representation of \cite{silverberg-05} consists of the last $n-1$ Weil restriction coordinates, together with three extra bits, say $0 \leq \nu \leq 3$ to resolve ambiguity in recovering the $x$-coordinate and $0 \leq \lambda \leq 1$ to determine the sign of the $y$-coordinate. So in our notation, Silverberg proposes to represent a point $(x,y) \in T_3$ is via the coordinates $(x_1,x_2,\nu,\lambda)$. The compression and decompression algorithms (in characteristic not equal to 3) carry out essentially the same steps:
\begin{itemize}
 \item Compute a univariate polynomial of degree $4$. The coefficients are polynomials over $\Fq$ in 2 indeterminates of degree at most 4.
 \item Compute the (up to 4) roots of this polynomial. During compression, this determines $\nu$. During decompression, $\nu$ determines which root is the correct one, and it is then used to compute $x_0$ via  addition and multiplication with constants.
 \item During decompression, compute the $y$-coordinate from the curve equation, using $\lambda$ to determine its sign. We disregard this step when estimating the complexity.
\end{itemize}
Since \cite{silverberg-05} does not contain a detailed analysis of the decompression algorithm, we cannot compare the exact number of operations. However, the essential difference with our approach is that Silverberg's compression and decompression algorithms both require computing the roots of a degree 4 polynomial over $\Fq$. For compression, this is clearly more expensive than our method, which consists only of evaluating some small expressions. For decompression, this is also less efficient than our method, which computes only a root of a quadratic polynomial, since running a root finding algorithm, or using explicit formulas for the solutions (i.e.\ solving the quartic by radicals), is much more complicated than computing the roots of our equation. 

One might argue that it is possible to represent $(x,y)$ via the
coordinates $(x_1,x_2)$ only. In such a case, compression would
consist simply of dropping $y$ and $x_0$ and would therefore have no
computational cost. Without remembering $\nu$ and $\lambda$ to resolve
ambiguity, this representation would identify up to 4 $x$-coordinates
and up to 8 full points. This is not much worse than our
representation, which identifies up to 3 $x$-coordinates and 6 full
points. However, it is not clear that this identification is
compatible with scalar multiplication of points. Therefore, one may
want to use at least $\nu$ to distinguish between the recovered
$x$-coordinates. This is in contrast with our situation, where we know
exactly which points are recovered during decompression (i.e.\ the
three Frobenius conjugates of the original point). Identifying these
three points is compatible with scalar multiplication, since $P =
\varphi^i(Q)$ implies $kP = \varphi^i(kQ)$ for all $k \in \N$ and $P,Q
\in T_3$, and so no extra bits are necessary. 

\bigskip \noindent
{\bf Comparison with Naumann's method.}
Naumann \cite{naumann} studies trace zero varieties for $n=3$. He does not give explicit compression and decompression algorithms, but he derives an equation for the trace zero subgroup that may be used for such. In fact, his equation is identical to our equation (\ref{eqn:naumann}), the Weil restriction of the (unsymmetrized) Semaev polynomial. However, he obtains it in a different way, namely, by eliminating from system (\ref{sys:frey}). 

Naumann suggests a compression method analogous to the one of Silverberg: A point is represented via the coordinates $(x_1,x_2,\nu,\lambda)$. For decompression, $x_0$ is recomputed from a quartic equation, $0 \leq \nu \leq 3$ determines which root of the equation is the correct $x_0$, and $0\leq \lambda \leq 1$ determines the sign of the $y$-coordinate. Hence the quartic equation must be solved during both compression and decompression. Naumann's equation is different from Silverberg's, yet the analysis of his method is analogous to that of Silverberg's method, and the conclusions are the same. In particular, his algorithms are less efficient than ours, and it is not clear whether it is possible to drop $\nu$ from the representation and still have a well defined scalar multiplication.

\bigskip \noindent
{\bf Security issues.} To the extent of our knowledge, there are no known attacks on the DLP in $T_3$ whose complexity is lower than generic (square root) attacks, provided that one chooses the parameters according to usual cryptographic practice. In particular, the group should have prime or almost prime order and be sufficiently large (e.g.\ 160 or 200 bits). We stress that index calculus methods, as detailed in \cite{gaudry-09} among many other works, do not yield an attack which is better than generic (square root) attacks in this setting, since the trace zero variety has dimension 2.

\section{Explicit equations for extension degree 5} \label{sec:eqnsdegree5}

The fifth Semaev polynomial is too big to be printed here, but a computer program can easily work with it. It has total degree 32 and degree 8 in each indeterminate. The symmetrized fifth Semaev polynomial has total degree 8 and degree 6 in the last indeterminate. In fact, it has degree 6 in the first, third and fifth indeterminate, and degree 8 in the second and fourth indeterminate. We can compute it efficiently with Magma. It has a small number of terms compared to the original polynomial, but printing it here would still take several pages.

The fact that we recover the missing coordinate from a degree 6 polynomial introduces some indeterminacy in the decompression process. However, extensive Magma experiments for different field sizes and curves show that for more than 90\% of all points in $T_5$, only a single class of Frobenius conjugates is recovered. For another 9\%, two classes (corresponding to $10$ $x$-coordinates) are recovered. Thus the ambiguity is very small for a great majority of points. In any case, this improves upon the approach of \cite{silverberg-05}, where the missing coordinate is recovered from a degree 27 polynomial, thus possibly yielding $27$ different $x$-coordinates.

The Weil restriction of the symmetric functions is
\begin{eqnarray*}
 s_1 & = & 5 x_0 \\
 s_2 & = & 10 x_0^2-5 \mu x_1x_4-5 \mu x_2x_3 \\
 s_3 & = & 10 x_0^3+5 \mu^2x_3^2x_4+5 \mu^2x_2x_4^2+5 \mu x_1x_2^2+5 \mu x_1^2x_3-15\mu x_0x_1 x_4-15 \mu x_0x_2x_3\\
 s_4 & = & 5 x_0^4-15\mu x_0^2x_1 x_4-15\mu x_0^2x_2x_3 -5 \mu x_1^3x_2-5 \mu^2x_1x_3^3-5 \mu^2x_2^3x_4-5 \mu^3x_3x_4^3+5 \mu^2x_2^2x_3^2\\
        &     & +5 \mu^2x_1^2x_4^2+10 \mu x_0x_1^2x_3+10 \mu x_0x_1x_2^2+10 \mu^2x_0x_3^2x_4+10 \mu^2x_0x_2x_4^2 -5 \mu^2x_1x_2x_3x_4\\
 s_5 & = & x_0^{5}+\mu^3x_3^{5}+\mu^4x_4^{5}+\mu x_1^5+\mu^2x_2^5-5 \mu^2x_1x_2^3x_3-5 \mu^3x_1x_2x_4^3-5 \mu^3x_2x_3^3x_4-5 \mu x_0x_1^3x_2\\
         &    & -5 \mu^2x_0x_1x_3^3-5 \mu^2x_0x_2^3x_4-5 \mu^3x_0x_3x_4^3-5 \mu^2x_1^3x_3x_4-5 \mu x_0^3x_1x_4-5 \mu x_0^3x_2x_3\\
         &    & +5 \mu x_0^2x_1^2x_3+5 \mu x_0^2x_1x_2^2+5 \mu^2x_0^2x_2x_4^2+5 \mu^2x_0^2x_3^2x_4+5 \mu^2x_0x_1^2x_4^2+5 \mu^2x_0x_2^2x_3^2\\
         &    & +5 \mu^2x_1^2x_2^2x_4+5 \mu^2x_1^2x_2x_3^2+5 \mu^3x_1x_3^2x_4^2+5 \mu^3x_2^2x_3x_4^2-5 \mu^2x_0x_1x_2x_3x_4.
\end{eqnarray*}
The compression algorithm computes $s_1,\ldots,s_4$ according to these formulas over $\Fq$. The decompression algorithm solves a degree 6 equation for $s_5$ and then recomputes the $x$-coordinate of the point. For the last step, we test two methods: We compute $x$ by factoring the polynomial $x^5 - s_1 x^4 + s_2x^3 - s_3x^2 + s_4x - s_5$ over $\F_{q^5}$, and we compute $x_0,\ldots,x_4$ by solving the above system over $\Fq$ with a Gr\"obner basis computation. Our experiments show that polynomial factorization can be up to 20 times as fast as computing a lexicographic Gr\"obner basis in Magma for some choices of $q$, and the entire decompression algorithm can be up to a factor 6 faster when implementing the polynomial factorization method. We give some exemplary timings for both methods for fields of 10, 20, 30, 40, 50 and 60 bits in Table \ref{timings5}. However, these experimental results can only be an indication: In Magma, the performance of the algorithms depends on the specific choice of $q$. In addition, any implementation exploiting a special shape of $q$ would most likely produce better results.

As for $n=3$, we suggest an equivalent representation $(t_1,t_2,t_3,t_4)$ where
\begin{equation} \label{tinx}
\begin{array}{rcl}
 t_1 & = & x_0\\
 t_2 & = & x_1x_4 + x_2x_3\\
 t_3 & = & x_1^2x_3 + x_1 x_2^2 + \mu x_3^2 x_4 + \mu x_2 x_4^2\\
 t_4 & = & \mu x_2^2 x_3^2 + \mu x_1^2 x_4^2 - \mu x_1 x_3^3 - x_1^3 x_2 - \mu x_2^3 x_4 - \mu^2 x_3 x_4^3 + \mu x_1 x_2 x_3 x_4\\
 t_5 & = & x_1^5 + \mu x_2 ^ 5 + \mu^2 x_3^5 + \mu^3 x_4^5 + 5 \mu x_1^2 x_2 x_3^2 + 5 \mu x_1^2 x_2^2 x_4 + 5 \mu^2 x_2^2 x_3 x_4^2 \\
       &     &+ 5 \mu^2 x_1 x_3^2 x_4^2 - 5 \mu x_1^3 x_3 x_4 - 5 \mu^2 x_2 x_3^3 x_4 - 5 \mu^2 x_1 x_2 x_4^3 - 5 \mu x_1 x_2^3 x_3
\end{array}
\end{equation}
and
\begin{equation} \label{sint}
 \begin{array}{rcl}
 s_1 & = &5 t_1 \\
 s_2 & = & 10 t_1^2 - 5 \mu t_2 \\
 s_3 & = & 10 t_1^3 - 15 \mu t_1t_2 + 5 \mu t_3\\
 s_4 & = & 5 t_1^4 - 15 \mu t_1^2t_2 + 10 \mu t_1 t_3 + 5 \mu t_4\\
 s_5 & = & t_1^5 - 5 \mu t_1^3 t_2 + 5 \mu t_1^2 t_3 + 5 \mu t_1 t_4 + \mu t_5.
\end{array} 
\end{equation}
Compared to the representation in the $s_i$, this representation gives a faster compression, but a slower decompression. Therefore, this approach may be useful in a setting where compression must be particularly efficient.

For decompression, the missing coordinate $t_5$ can be recomputed from a degree 6 equation, which we obtain by substituting the relations (\ref{sint}) into the symmetrized fifth Semaev polynomial. Afterwards we may either recompute $s_1,\ldots,s_5$ from $t_1,\ldots,t_5$ according to system (\ref{sint}) and solve $x^5 - s_1 x^4 + s_2x^3 - s_3x^2 + s_4x - s_5$ for $x$, or else we may solve system (\ref{tinx}) directly for $x_0,\ldots,x_4$ with Gr\"obner basis techniques. The polynomial factorization method is equivalent to using the representation in the $s_i$, only that some of the computations are shifted from the compression to the decompression algorithm. The Gr\"obner basis method (use $t_i$ and compute Gr\"obner basis, ``second method'') compares to using $s_i$ with Gr\"obner basis (``first method'') as given in Table \ref{timings5}. We see that the second method is a factor 2 to 3 faster in compression, but slower in decompression. The reason for this is that the polynomial used to recompute the missing coordinate is more complicated for the second method, and evaluation of polynomials is quite slow in Magma. Solving for the missing coordinate takes 5 times longer for the second method. The solution of system (\ref{decsys}), which we achieve by computing a lexicographic Gr\"obner basis and solving the resulting triangular system in the obvious way, takes the same amount of time in both cases.

\begin{table}\footnotesize
\caption{Average time in milliseconds for compression/decompression of one point when $n=5$}
\label{timings5}
\begin{tabular}{l|llllll}
\hline\noalign{\smallskip}
$q$ & $2^{10}-3$ & $2^{20}-5$ & $2^{30}-173$ & $2^{40}-195$ & $2^{50}-113$ & $2^{60}-695$ \\
\noalign{\smallskip}\hline\noalign{\smallskip}
Compression $s_i$                                                  & 0.041 & 0.048 & 0.052 & 0.106 & 0.108 & 0.112 \\
Compression $t_i$                                                   & 0.017 & 0.022 & 0.024 & 0.031 & 0.021 & 0.048 \\
Decompression $s_i$ poly factorization & 5.536 & 16.480 & 21.423 & 45.080 & 55.872 & 59.520 \\
Decompression $s_i$ Gr\"obner basis                & 24.134 & 26.470 & 39.593 & 101.559 & 104.490 & 118.991 \\
Decompression $t_i$ Gr\"obner basis                 & 38.375 & 40.198 & 60.438 & 132.484 & 133.088 & 150.083 \\
\noalign{\smallskip}\hline
\end{tabular}
\end{table}

We now give an example of our compression/decompression algorithms, including two points $P$ on the trace zero variety where $\decompr(\compr(P))$ produces the minimum and maximum possible number of outputs.

\begin{example} \label{ex5}
Let $E$ be the curve $y^2 = x^3 + x + 135$ over $\Fq$, where $q = 2^{60}-695$ is a $60$-bit prime, and $\mu = 3$. The trace zero subgroup of $E(\F_{q^5})$ has prime order of 240 bits. We choose a random point 
{\small 
\begin{eqnarray*}
P & = & 697340666673436518 + 801324486821916366 \zeta + 191523769921581598 \zeta^2 \\
&&+ 193574581008452232 \zeta^3 + 808272437423069772 \zeta^4    \in T_5
\end{eqnarray*}}
and compute
$${\small 
\begin{array}{l}
 \compr(P)  =  (27938819546643747, 599177118073319826, 587362643323803394, 899440023033601132) \smallskip \\
 \decompr(27938819546643747, 599177118073319826, 587362643323803394, 899440023033601132)\\
 ~~~ =  \{ 697340666673436518 + 801324486821916366 \zeta + 191523769921581598 \zeta^2 \\
 ~~~~~~~~~~  + 193574581008452232 \zeta^3 + 808272437423069772 \zeta^4, \\
 ~~~~~~~~~~     697340666673436518 + 836712212802745328 \zeta + 506907366758395901 \zeta^2  \\
  ~~~~~~~~~~   + 517000572714098077 \zeta^3 + 268866625974497959  \zeta^4, \\
 ~~~~~~~~~~   697340666673436518  + 960543166171367987 \zeta+ 126552294958642222 \zeta^2\\
    ~~~~~~~~~~ +  448251978051599093 \zeta^3  + 74315924307841334 \zeta^4  , \\
 ~~~~~~~~~~   697340666673436518 + 810370833605859760 \zeta + 539948230971075773 \zeta^2 \\
   ~~~~~~~~~~  + 1032750511909194579 \zeta^3 + 944608723064092684 \zeta^4 , \\
 ~~~~~~~~~~  697340666673436518  + 49813814418649402 \zeta+ 940911346603997068 \zeta^2\\
    ~~~~~~~~~~ + 114265365530348581 \zeta^3+ 209779298444190813 \zeta^4  \}.
\end{array}}$$
When using the second variant of the representation, we compute
$${\small
\begin{array}{c} (t_1,t_2,t_3,t_4) = (697340666673436518, 553115374027544004, 315951679773440541, 285024754797056479). 
\end{array}}$$
For this point, the results of decompression are exactly the Frobenius conjugates of $P$. However, this is not always the case. In rare cases, the algorithm may recover up to six classes of Frobenius conjugates. We give an example of a point for which three classes of Frobenius conjugates are recovered:
{\small 
\begin{eqnarray*}
P & = & 760010909342414570 + 568064535058825884 \zeta+ 244006548504894796 \zeta^2\\
&&+ 446522043528586762 \zeta^3  + 731314735984238952 \zeta^4  \in T_5.
\end{eqnarray*}}
\end{example}

\medskip \noindent
{\bf Operation count for representation in the $s_i$.} 
Given $x_0,\ldots,x_4$, the numbers $t_1,\ldots,t_4$ can be computed with a total of 5S+13M according to (\ref{tinx}). Then $s_1,\ldots, s_4$ can be computed from those numbers with 2S+3M as given in (\ref{sint}). This seems to be the best way to compute $s_1,\ldots,s_4$, since these formulas group the terms that appear several times. Hence compression takes a total of 7S+16M.

 For decompression, the most costly part of the algorithm is factoring the polynomials. First, the algorithm has to factor a degree 6 polynomial over $\Fq$, and next, a degree 5 polynomial over $\F_{q^5}$. The asymptotic complexity for both of these is $O(\log q)$ operations in $\Fq$.

\bigskip \noindent
{\bf Operation count for representation in the $t_i$.} Compression takes 5S+13M. For decompression, we can either recompute $s_1,\ldots,s_5$ from $t_1,\ldots,t_5$ and factor the polynomial, in which case this approach is exactly the same as the above. Or else we can solve system (\ref{tinx}) by means of a Gr\"obner basis computation over $\Fq$. Since there are no practically meaningful bounds for Gr\"obner basis computations, a complexity analysis of this approach makes no sense.

\bigskip \noindent
{\bf Comparison with Silverberg's method.}
Concrete equations are presented in \cite{silverberg-05} for the case where the ground field has characteristic 3. The most costly parts of the compression and decompression algorithms are computing the resultant of two polynomials of degree 6 and 8 with coefficients in $\Fq$, and finding the roots of a degree 27 polynomial over $\Fq$. In general, resultant computations are difficult, and the polynomial to be factored has much larger degree than those in our algorithm. In Silverberg's approach, five extra bits are required to distinguish between the possible 27 roots of the polynomial.

Although neither Silverberg nor we give explicit equations for larger $n$, our understanding is that our algorithm scales better with increasing $n$, since our method is more natural and respects the structure of the group.

\bigskip \noindent
{\bf Security issues.} We briefly discuss the security issues connected with use of $T_5$ in 
DL-based and pairing-based cryptosystems. 

Since $T_5$ is a group of size $q^4$, generic
algorithms that solve the DLP in $T_5$ have complexity $O(q^2)$. 
Security threats in the context of DL-based cryptosystems are posed by 
algorithms for solving the DLP that achieve lower complexity. 
There are two types of algorithms that one needs to consider: 
First, cover attacks aim to transfer the DLP in $E(\F_{q^5})$ to the DLP in the
Picard group of a curve of larger (but still rather low) genus, see
\cite{ghs,diem-ghs}. The DLP is then solved there using index calculus
methods. Combining the results of \cite{diem-ghs} 
and \cite{diem-06}, it is sometimes possible to map the DLP from $T_5$ into the  Picard group of a
genus 5 curve (which is usually not hyperelliptic), where it can be
solved with probabilistic complexity $\tilde{O}(q^{4/3})$ following the approach of~\cite{diem-kochinke}. 
However, only a very small proportion of curves is affected by this attack, 
and such curves should be avoided in practice. Moreover, in order to avoid isogeny attacks, the curve should be chosen such that 4 does not divide the order of $T_5$, see \cite{diem-ghs}. 
Second, the index calculus attack of \cite{gaudry-09} applies to $T_5$ and has complexity $\tilde{O}(q^{3/2})$.  
This makes $T_5$ not an ideal group to use in a DL-based cryptosystem.
Notice that in practice, however, the constant in the $O$ is very large, 
since the attack requires Gr\"obner basis computations, which are very time consuming 
(their worst case complexity is doubly exponential in the size of the input), 
and often do not terminate in practice. 
It is our impression that more in depth study is needed in order to give a
precise estimate of the feasibility of such an attack for a practical
choice of the parameters. We carried on preliminary experiments, 
which indicate that a straightforward application of the method from \cite{gaudry-09} to $T_5$ 
yields a system of equations which is very costly to compute (it requires computing the Weil descent of the fifth Semaev polynomial) and which Magma cannot solve in several weeks and using more than 300 GB of memory on the same machine that we used to carry out the experiments reported on in Sections \ref{sec:eqnsdegree3} and \ref{sec:eqnsdegree5} of this article.
Notice that solving such a system would (possibly) produce one relation, to be then used in an index calculus 
attack. Therefore, in practice one would need to solve many such systems, in order to produce the relations needed 
for the linear algebra step of the index calculus attack.

Trace zero varieties are even more interesting in the context of pairing-based cryptography. 
The main motivation comes from~\cite{rubin-silverberg-09}, where Rubin and Silverberg show that supersingular 
abelian varieties of dimension greater than one offer more security than supersingular elliptic curves, 
for the same group size. Trace zero varieties are explicitly mentioned in~\cite{rubin-silverberg-09} as one of the 
most relevant examples of abelian varieties for pairing-based cryptography. In order to estimate the security of $T_5$ 
in pairing-based cryptosystems, one needs to compare the complexities of solving the DLP in $T_5$ and in $\F_{q^{5k}}$, 
where $k$ is the embedding degree, i.e., the smallest integer $k$ such that $\F_{q^{5k}}$ contains the image of the pairing.
A first observation is that, since the results of \cite{rubin-silverberg-09} hold over fields of any characteristic, one should avoid fields of 
small characteristic, so that the recent attacks from \cite{gologlu-granger-mcguire-zumbragel-13-1,joux-13,gologlu-granger-mcguire-zumbragel-13-2,caramel-13,barbulescu-gaudry-joux-thome-13} do not apply.
Over a field of large characteristic, the cover and index calculus attacks that we discussed in the previous paragraph 
do not seem to pose a serious security concern in the context of pairing-based cryptography. 
This is due to the fact that, for most supersingular elliptic curves, the Frey-R\"uck or the MOV attack have 
lower complexity than cover and index calculus attacks in the lines of ~\cite{gaudry-09,ghs,diem-ghs,diem-kochinke}. 
In some cases however, the choice of the security parameter may need to be adjusted, according to the complexity of 
these index calculus attacks.
As an example, let us discuss the choice of parameters for a pairing with 80-bits security. 
One needs a field of about 1024-bits as the target of the pairing (avoiding fields of small characteristic). 
If we assume that the pairing ends up in an extension field of degree $k=2$ of the original field $\F_{q^5}$ (this is the case 
for most supersingular elliptic curves), then $q$ should be a 102-bit number. A $q^{3/2}$ attack on the group $T_5$ 
on which the pairing is defined would result in 153-bit security, while a $q^{4/3}$ attack would result in 136-bit security. 
However, on the side of the finite field the system has an 80-bit security, so the attacks 
from~\cite{gaudry-09,ghs,diem-ghs,diem-kochinke} end up not  influencing the overall security of the pairing-based 
cryptosystem in this case.
A related comment is that an interesting case for pairings is when the DLP in $T_5$ and in the finite field extension 
$\F_{q^{5k}}$ where the pairing maps have the same complexity. In order to achieve this in our previous example, 
we would need to have a security parameter $k = 4$, which can be achieved by supersingular trace zero varieties. 
In this case, the complexity of solving a DLP in $T_5$ and in $\F_{q^{20}}$ are both 
about 80-bits when $q\sim 2^{53}$. 
Summarizing, the complexity of the DLP in $T_5$ coming from the works~\cite{gaudry-09,ghs,diem-ghs,diem-kochinke} 
influences the choice of the specific curves that we use in pairing-based applications, 
since it influences the security parameter $k$ that makes the hardness of solving the DLP in $T_5$ and in 
$\F_{q^{5k}}$ comparable, and the value of $k$ depends on the choice of the curve. However, in general 
it does not influence the size $q$ of the field that we work on, since an attack can influence the value of $q$ only if 
it has lower complexity than the Frey-R\"uck or the MOV attack for supersingular elliptic curves. Therefore, 
using trace zero varieties instead of elliptic curve groups in pairing-based cryptography has the advantages of 
enhancing the security and allowing for more flexibility in the setup of the system.

\section{Conclusion} \label{sec:conclusion}

The Semaev polynomials give rise to a useful equation describing the
$\Fq$-rational points of the trace zero variety. Its significance is
that it is one single equation in the $x$-coordinates of the elliptic
curve points, but unfortunately its degree grows quickly with $n$. Using this
equation, we obtain an efficient method of point compression and
decompression. It computes a representation for the $\Fq$-points of
the trace zero variety that is optimal in size for $n=3$ and for $n=5$. Our polynomials have lower degree than
those used in the representations of \cite{silverberg-05} (1 compared
to 4 for $n=3$, and 6 compared to 27 for $n=5$) and \cite{naumann} (1 compared to 3 for $n=3$), thus allowing more
efficient compression and decompression and less ambiguity in the recovery process. Finally, our representation is interesting from a mathematical point of view, since it is the first representation (to our knowledge) that is compatible with scalar multiplication of points.

\vspace{.6cm}\noindent\small {\bf Acknowledgements}\quad
We thank Pierrick Gaudry and Peter Schwabe for helpful discussions and Tanja Lange for pointing out the work of Naumann. We are grateful to the mathematics department of the Univerity of Z\"urich for access to their computing facilities. The authors were supported by the Swiss National Science Foundation under Grant No.\ 123393.

\bibliographystyle{plain}
\bibliography{lit}

\end{document}